\documentclass[journal,article,submit,moreauthors,pdftex,12pt,a4paper,axioms]{mdpi}

  \def\sw#1{{\sb{(#1)}}} 
   
  \def\su#1{{\sp{[#1]}}}

  \def\<{{\langle}} 
  \def\>{{\rangle}}

  \def\eps{\varepsilon}

  \def\note#1{{}} 
 \def\can{{\rm \textsf{can}}}

  \def\note#1{} 

  \def\cO{{\mathcal O}}

  \def\rhom#1#2#3{{{\rm Hom}\sb{#1}(#2,#3)}}

  \def\beq{\begin{equation}} 
  \def\eeq{\end{equation}}

  \def\id{\mathrm{id}} 
  \def\im{{\rm Im}}

  \def\ut{{\otimes}} 
  \def\ot{{\otimes}}

  \def\roA{{\varrho^A}}

 \def\coker{\mathrm{coker}}

\def\stac#1{\raise-.2cm\hbox{$\stackrel{\displaystyle\otimes}{\scriptscriptstyle{#1}}$}}

\def\cten#1{\raise-.2cm\hbox{$\stackrel{\displaystyle\widehat{\otimes}}
{\scriptscriptstyle{#1}}$}}

\lastpage{x}
\doinum{10.3390/------}
\pubvolume{xx}
\pubyear{2012}
\history{Received: xx / Accepted: xx / Published: xx}

  \def\Label#1{\label{#1}\ifmmode\llap{[#1] }\else 
  \marginpar{\smash{\hbox{\tiny [#1]}}}\fi} 
  \def\Label{\label}

\def\ot{\otimes}

\def\CC{{\mathbb C}}

\def\NN{{\mathbb N}}

\def\PP{{\mathbb P}}

\def\RR{{\mathbb R}}

\def\WW{{\mathbb W}}

\def\ZZ{{\mathbb Z}}

\newcommand{\Cc}{\mathcal{C}}

\newcommand{\Hh}{\mathcal{H}}

\def\*C{{}^*\hspace*{-1pt}{\Cc}}

\def\text#1{{\rm {\rm #1}}}

 \def\1{\mathbf{1}}

      \def\tr{\mathrm{Tr}\ }

  \usepackage[all]{xy}
  \usepackage{amsfonts} 
  \usepackage{amsmath} 
  \xyoption{2cell}

\newtheorem{proposition}{Proposition}[section]
  \newtheorem{lemma}[proposition]{Lemma} 
   
  \newtheorem{theorem}[proposition]{Theorem} 

  \newtheorem{definition}[proposition]{Definition}
    
  \newtheorem{example}[proposition]{Example}

\Title{Bundles over Quantum Real Weighted Projective Spaces} 
\Author{Tomasz Brzezi\'nski$^*$ and Simon A.\ Fairfax}
 \address{ Department of Mathematics, Swansea University, 
  Singleton Park,  Swansea SA2 8PP, U.K.} 
  \corres{T.Brzezinski@swansea.ac.uk}

  \abstract{ The algebraic approach to bundles in non-commutative geometry and the definition of quantum real weighted projective spaces are reviewed.  Principal $U(1)$-bundles over quantum real weighted projective spaces are constructed. As the spaces in question fall into two separate classes, the {\em negative} or {\em odd} class that generalises quantum real projective planes and the {\em positive} or {\em even} class that generalises the quantum disc, so do the constructed principal bundles. In the negative case the principal bundle is proven to be non-trivial and associated projective modules are described. In the positive case the principal bundles turn out to be trivial, and so all the associated modules are free. It is also shown that the circle (co)actions on the quantum Seifert manifold that define quantum real weighted projective spaces are almost free.} 

\begin{document}

\section{Introduction.}

In an algebraic setup an action of a circle on a quantum space corresponds to a coaction of a Hopf algebra of Laurent polynomials in one variable on the noncommutative coordinate algebra of  the quantum space. Such a coaction can equivalently be understood as a $\ZZ$-grading of this coordinate algebra. A typical $\ZZ$-grading assigns degree $\pm 1$ to every generator of this algebra (different from the identity). The degree zero part forms a subalgebra which in particular cases corresponds to quantum complex or real projective spaces (grading of coordinate algebras of quantum spheres \cite{VakSob:alg} or prolonged quantum spheres \cite{BrzZie:pri}). Often this grading is strong, meaning that the product of $i,j$-graded parts is equal to the $i+j$-part of the total algebra. In geometric terms this reflects the freeness of the circle action. 
 
In two recent papers \cite{BrzFai:tea} and \cite{Brz:sei} circle actions on three-dimensional (and, briefly, higher dimensional) quantum spaces were revisited. Rather than assigning a uniform grade to each generator, separate generators were given degree by pairwise coprime integers.  The zero part of  such a grading  of the coordinate algebra of the quantum odd-dimensional sphere corresponds to the quantum weighted projective space, while the zero part of such a grading of the algebra of the prolonged even dimensional quantum sphere leads to quantum real weighted projective spaces. 

In this paper we focus on two classes of algebras $\cO(\RR\PP_q^2(l;-))$ ($l$ a positive integer) and $\cO(\RR\PP_q^2(l;+))$ ($l$ an odd positive integer) identified in \cite{Brz:sei} as fixed points of weighted circle actions on the coordinate algebra $\cO(\Sigma_q^3)$ of a non-orientable quantum Seifert manifold described in \cite{BrzZie:pri}. Our aim is to construct quantum $U(1)$-principal bundles over the corresponding quantum spaces $\RR\PP_q^2(l;\pm)$ and describe associated line bundles. Recently, the importance of such bundles in non-commutative geometry  was once again brought to the fore in \cite{BegBrz:lin}, where the non-commutative Thom construction was outlined.  As a further consequence of the principality of  $U(1)$-coactions we also deduce that $\RR\PP_q^2(l;\pm)$ can be   understood as quotients of $\Sigma_q^3$ by almost free $S^1$-actions. 

We begin in Section~2 by reviewing elements of algebraic approach to classical and quantum bundles. We then proceed to describe algebras $\cO(\RR\PP_q^2(l;\pm))$ in Section~3. Section~4 contains main results including construction of principal comodule algebras over $\cO(\RR\PP_q^2(l;\pm))$. We observe that constructions albeit very similar in each case yield significantly different results. The principal comodule algebra over $\cO(\RR\PP_q^2(l;-))$ is non-tirivial while that over $\cO(\RR\PP_q^2(l;+))$ turns out to be trivial (this means that all associated bundles are trivial, hence we do not mention them in the text). Whether it is a consequence of our particular construction or there is a deeper (topological or geometric) obstruction to constructing non-trivial principal circle bundles over $\RR\PP_q^2(l;+)$ remains an interesting open question.

Throughout we work with involutive algebras over the field of complex numbers (but the algebraic results remain true for all fields of characteristic 0). All algebras are associative and have identity, we use the standard Hopf algebra notation and terminology and we always assume that the antipode of a Hopf algebra is bijective. All topological spaces are assumed to be Hausdorff.

\section{Review of bundles in non-commutative geometry.}\label{sec.rev}

The aim of this section is to set-out the topological concepts in relation to topological bundles, in particular principal bundles. The classical connection is made for interpreting topological concepts in an algebraic setting, providing a manageable methodology for performing calculations. In particular, the connection between principal bundles in topology and the algebraic Hopf-Galois condition is described. 

\subsection{Topological aspects of bundles.}

As a natural starting point, bundles are defined and topological properties are described. The principal map is defined and shown that injectivity is equivalent to the freeness condition. The image of the canonical map is deduced and necessary conditions are imposed to ensure the bijectivity of this map. The detailed account of the material presented in this section can be found in \cite{BauHaj:non}.

\begin{definition}\rm
A {\em bundle} is a triple $(E,\pi,M)$ where $E$ and $M$ are topological spaces and $\pi:E \rightarrow M$ is a continuous surjective map. Here $M$ is called the {\em base space}, $E$ the {\em total space} and $\pi$ the {\em projection} of the bundle.
\end{definition}

For each $m \in M$, the {\em fibre} over $m$ is the topological space $\pi^{-1}(m)$, i.e.\ the points on the total space which are projected, under $\pi$, onto the point $m$ in the base space. A bundle whose fibres are homeomorphic which satisfies a condition known as local triviality are known as fibre bundles. This is formally expressed in the next definition.

\begin{definition}\rm 
A {\em fibre bundle} is a triple $(E, \pi, M, F)$ where $(E, \pi, M)$ is bundle and $F$ is a topological space such that $\pi^{-1}(m)$ are homeomorphic to $F$ for each $m \in M$. Furthermore, $\pi$ satisfies the local triviality condition.
\end{definition}

The local triviality condition is satisfied if for each $x \in E$, there is an open neighourhood $U \subset B$ such that $\pi^{-1}(U)$ is homeomorphic to the product space $U \times F$, in such a way that $\pi$ carries over to the projection onto the first factor. That is the following diagram commutes:
$$
\xymatrix {{\pi^{-1}(U)} \ar[d]_{\pi} \ar[rr]^{\phi} && U \times F \ar[dll]^{p_1} \\
 U.}
$$
The map $p_1$ is the natural projection $U \times F \rightarrow U$ and $\phi : \pi^{-1}(U) \rightarrow U \times F$ is a homeomorphism.

\begin{example}\rm  An example of a fibre bundle which is non-trivial, i.e not a global product space, is the M\"obius strip. It has a circle that runs lengthwise through the centre of the strip as a base B and a line segment running vertically for the fibre F. The line segments are in fact copies of the real line, hence each $\pi^{-1}(m)$ is homeomorphic to $\RR$ hence the Mobius strip is a fibre bundle.
\end{example}

Let $X$ be a topological space which is compact and satisfies the Hausdorff property and G a compact topological group. Suppose there is a right action $\triangleleft :X \times G \rightarrow X$ of $G$ on $X$  and write  $x \triangleleft g=xg$.

\begin{definition} \rm
An action of $G$ on $X$ is said to be {\em free} if $x g=x$ for any $x \in X$ implies that $g=e$, the group identity.
\end{definition}

With an eye on algebraic formulation of freeness, the {\em principal map} $F^G:X \times G \rightarrow X \times X$ is defined as $(x,g) \mapsto (x,xg)$.

\begin{proposition} \label{inj map}
$G$ acts freely on $X$ if and only if $F^G$ is injective.
\end{proposition}
\begin{proof}
``$\Longleftarrow$" Suppose the action is free, hence $xg=x$ implies that $g=e$. If $(x,xg)=(x',x'g')$, then $x=x'$ and $xg=xg'$. Applying the action of $g'^{-1}$ to both sides of $xg=xg'$ we get $x(gg'^{-1})=x$, which implies $gg'^{-1}=e$ by the freeness property, concluding $g=g'$ and $F^G$ is  injective as required.

``$\Longrightarrow$" Suppose $F^G$ is injective, so $F^G(x,g)=F^G(x',g')$ or $(x,xg)=(x',x'g')$ implies $x=x'$ and $g=g'$. Since $x=xe$ from the properties of the action, if $x=xg$ then $g=e$ from the injectivity property.
\end{proof}

Since $G$ acts on $X$ we can define the quotient space $X/G$,
$$
Y=X/G:=\{ [x]:x \in X \}, \qquad \mbox{where} \qquad [x]=xG=\{xg: g \in G \}.
$$
The sets $xG$ are called the {\em orbits} of the points $x$. They are defined as the set of elements in $X$ to which $x$ can be moved by the action of elements of $G$. The set of orbits of $X$ under the action of $G$ forms a partition of $X$, hence we can define the equivalence relation on $X$ as,
$$
x \sim y \Longleftrightarrow \exists g \in G \ \ \mbox{such that} \ \ xg=y.
$$ 
The equivalence relation is the same as saying $x$ and $y$ are in the same orbit, i.e., $xG=yG$. Given any quotient space, then there is a canonical surjective map 
$$
\pi:X \rightarrow Y=X/G, \qquad x \mapsto xG=[x],
$$
which maps elements in $X$ to the their class of orbits. We define the pull-back along this map $\pi$ to be the set
$$
X \times_Y X:=\{(x,y) \in X \times X: \pi(x)=\pi(y) \}.
$$
As described above, the image of the principal map $F^G$ contains elements of $X$ in the first leg and the action of $g \in G$ on $x$ in the second leg. To put it another way, the image records elements of $x \in X$ in the first leg and all the elements in the same orbit as this $x$ in the second leg. Hence we can identify the image of the canonical map as the pull back along $\pi$, namely $X \times_Y X$. This is formally proved as a part of the following proposition.

\begin{proposition}\label{prop free class}
$G$ acts freely on $X$ if and only if the map 
$$
F_X^G:X \times G \rightarrow X \times_Y X, \qquad (x,g) \mapsto (x,xg),
$$ 
is bijective. 
\end{proposition}
\begin{proof}
First note that the map $F_X^G$ is well-defined since the elements $x$ and $xg$ are in the same orbit hence map to the same equivalence class under $\pi$. Using Proposition~\ref{inj map} we can deduce that the injectivity of $F_X^G$ is equivalent to the freeness of the action. Hence if we can show that $F_X^G$ is surjective the proof is complete.

Take $(x,y) \in X \times_Y X$.  This means $\pi(x)=\pi(y)$ which implies $x$ and $y$ are in the same equivalence class, which in turn means they are in the same orbit. We can therefore deduce that $y=xg$ for some $g \in G$. So, $(x,y)=(x,xg)=F_X^G(x,g)$ implying $(x,y) \in \im F_X^G$. Hence $\im F_X^G=X \times_Y X$ completing the proof.
\end{proof}

\begin{definition} \label{prin act def} \rm An action of $G$ on $X$ is said to be {\em principal} if the map $F^G$ is both injective and proper, i.e.\ it is injective, continuous and such that the inverse image of a compact subset is compact. 
\end{definition}

Since the injectivity and freeness condition are equivalent we can interpret principal actions as both free and proper actions. We can also deduce that these types of actions give rise to homeomorphisms $F_X^G$ from $X \times G$ onto the space $X \times_{X/G} X$. Principal actions lead to the concept of topological principle bundles.

\begin{definition} \rm
A {\em principal bundle} is a quadruple $(X,\pi,M,G)$ such that 

(a) $(X,\pi,M)$ is a bundle and $G$ is a topological group acting continuously on $X$ with action $\triangleleft:X \times G \rightarrow X$, $x \triangleleft g=xg$; 

(b) the action $\triangleleft$ is principal; 

(c) $\pi(x)=\pi(y) \Longleftrightarrow \exists g \in G$ such that $y=xg$; 

(d) the induced map $X/G \rightarrow M$ is a homeomorphism.
\end{definition}

The first two properties tell us that principal bundles are bundles admitting a principal action of a group $G$  on the total space $X$, i.e.\ principal bundles correspond to principal actions. By Definition $(\ref{prin act def})$, principal actions occur when the principal map is both injective and proper, or equivalently, when the action is free and proper. The third property ensures that the fibres of the bundle correspond to the orbits coming from the action and the final property implies that the quotient space can topologically be viewed as the base space of the bundle.

\begin{example}\rm 
Suppose $X$ is a topoplogical space and $G$ a topological group which acts on $X$ from the right. The triple $(X, \pi, X/G)$ where $X/G$ is the orbit space and $\pi$ the natural projection is a bundle. A principal action of $G$ on $X$ makes the quadruple $(X,\pi,X/G,G)$ a principal bundle. 
\end{example} 

We describe a principal bundle $(X, \pi, Y,G)$ as a $G$-principal bundle over $(X, \pi, Y)$, or $X$ as a $G$-principal bundle over $Y$.

\begin{definition}\rm 
A {\em vector bundle} is a bundle $(E,\pi,M)$ where each fibre $\pi^{-1}(m)$ is endowed with a vector space structure such that addition and scalar multiplication are continous maps.
\end{definition}

Any vector bundle can be understood as a bundle associated to a principal bundle in the following way. Consider a $G$-principal bundle $(X, \pi, Y,G)$ and let $V$ be a representation space of $G$, i.e.\ a (topological) vector space with a (continuous) left $G$-action $\triangleright: G\times V\to V$, $(g,v)\mapsto g\triangleright v$.  Then $G$ acts from the right on $X\times V$ by
$$
(x,v)\triangleleft g := (xg, g^{-1}\triangleright v), \quad \mbox{for all $x\in X$, $v\in V$ and $g\in G$}. 
$$
We can define
$
E = (X\times V)/G,
$
and  a surjective (continuous map)
$
\pi_E: E\to Y$, $(x,v)\triangleleft G\mapsto \pi(x),
$
and thus have a fibre bundle $(E,\pi_E,Y,V)$. In the case where $V$ is a vector space, we assume that $G$ acts linearly on $V$.

\begin{definition}  \rm A {\em section} of a bundle $(E,\pi_E, Y)$ is a
continuous map $s: Y\to E$ such that, for all $y\in Y$,
$$
\pi_E (s(y)) = y,
$$
i.e.\ a section is simply a section of the morphism $\pi_E$. The set of sections of $E$ is denoted by $\Gamma(E)$.
\end{definition}

\begin{proposition} \label{sec.as}   Sections in a fibre bundle $(E,\pi_E,Y,V)$ associated to a principal $G$-bundle $X$ are in bijective correspondence with (continuous) maps $f: X\to V$ such that
$$
f(xg) = g^{-1}\triangleright f(x).
$$
All such $G$-equivariant maps are denoted by $\rhom G X V$.
\end{proposition}
\begin{proof}
Remember that $Y=X/G$. Given a map $f\in \rhom G X V$, define the section
$
s_f : Y\to E$ , $xG\mapsto (x,f(x))\triangleleft G.$

Conversely, given $s\in \Gamma(E)$, define
$f_s: X\to V$ by assigning to $x\in X$ a unique $v\in V$ such that $s(xG) = (x,v)\triangleleft G$. Note that $v$ is unique, since if $(x,w) = (x,v)\triangleleft g$, then $xg=x$ and $w = g^{-1}\triangleright v$. Freeness  implies that $g=e$, hence $w=v$. The map $f_s$ has the required equivariance property, since the element of $(X\times V)/G$ corresponding to $xg$ is $g^{-1}\triangleright v$.
\end{proof}

\subsection{Non-commutative principal and associated bundles.}\label{sec.non.pri}

Consider (complex) algebras $\cO(X)$, $\cO(Y)$ and $\cO(G)$  of  coordinate functions on compact spaces $X,Y$ and $G$ considered in the previous subsection. Put $A=\cO(X)$ and $H=\cO(G)$ and note the identification $\cO(G \times G) \cong \cO(G) \otimes \cO(G)$. Through this identification, $\cO(G \times G)$ is a Hopf algebra with 
{comultiplication:} $f \mapsto (\Delta f),$ $(\Delta f)(g,h)=f(gh)$, counit $\eps: \cO(G) \rightarrow \CC$, $\eps(f)=f(e)$, and the antipode $
S:H \rightarrow H$, $(Sf)(g)=f(g^{-1})$.

Using the fact that $G$ acts on $X$ we can construct a right coaction of $H$ on $A$
  by $
\roA:A \rightarrow A \otimes H$, $\roA(a)(x,g)=a(xg)$. This coaction is an algebra map due to the commutativity of the algebras of functions involved.

We have viewed the spaces of functions on  $X$ and $G$, next we view the space of functions on Y,  $B:=\cO(Y)$, where $Y=X/G$.  $B$ is a sub-algebra of $A$ by 
$$
\pi^*:B \rightarrow A, \qquad b \mapsto b \circ \pi,
$$
where $\pi$ is the canonical surjection defined above. The map $\pi^*$ is injective, since $b \neq b'$ in $\cO(X/G)$ means there exists at least one orbit $xG=[x]$ such that $b([x]) \neq b'([x])$, but $\pi(x)=[x]$, so $b(\pi(x)) \neq b'(\pi(x))$ which implies $\pi^*(b) \neq \pi^*(b')$. Therefore, we can identify $B$ with $\pi^*(B)$. Furthermore, $a \in \pi^*(B)$ if and only if 
$$a(xg) = a(x),$$
 for all $x\in X$, $g\in G$. This is the same as
 $$
 \roA(a)(x,g) = (a\ot 1)(x,g),
 $$
 for all $x\in X$, $g\in G$, where $1: G\to \CC$ is the unit function $1(g) =1$ (the identity element of $H$). Thus we can identify $B$ with the {\em coinvariants} of the coaction $\roA$:
 $$
 B = A^{coH} := \{ a\in A\; |\; \roA(a) = a\ot 1\}.
 $$

Since $B$ is a subalgebra of $A$, it acts on $A$ via the inclusion map
$
(ab)(x) = a(x)b(\pi(x))$, $(ba)(x) = b(\pi(x))a(x)$. 
We can identify $\cO(X\times_Y X)$ with $\cO(X)\ot_{\cO(Y)}\cO(X) = A\ot_B A$ by the map
$$
\theta (a\ot_B a')(x,y) = a(x)a'(y),\quad  \mbox{(with $~\pi(x) = \pi(y)$).}
$$
Note that $\theta$ is well defined because $\pi(x) =\pi(y)$. Proposition~\ref{prop free class} immediately yields

\begin{proposition}\label{prop free quan}
The action of $G$ on $X$ is free if and only if  $F_X^{G*}: \cO(X \times_Y X) \rightarrow \cO(X \times G)$, $f \mapsto f \circ F_X^{G}$ is bijective.
\end{proposition}

 In view of the  definition of the coaction of $H$ on $A$, we can identify 
${F_X^G}^*$ with the {\em canonical map}
$$
\can: a\ot_Ba'\mapsto [(x,g)\mapsto a(x)a'(x.g)] = a\roA(a').
$$
Thus the action of $G$ on $X$ is free if and only if this purely algebraic map is bijective. In the classical geometry case we take  $A=\cO(X)$, $H =\cO(G)$ and $B=\cO(X/G)$, but in general there is no need to restrict oneself to commutative algebras (of functions on topological spaces). In full generality this leads to the following definition.

\begin{definition}\label{def.H=G} \rm (Hopf-Galois Extensions) Let $H$ be a Hopf algebra and $A$ a right $H$-comodule algebra with coaction given by $\roA:A \rightarrow A \otimes H$. Define $B:= \{ b \in A \; |\; \roA(b)=b \otimes 1 \}$. We say that $B \subseteq A$ is a {\em Hopf-Galois extension} if the left $A$-module, right $H$-comodule map
$$
\can: A \otimes_B A \rightarrow A \otimes H, \qquad a \otimes_B a' \mapsto a\roA( a')
$$
is an isomorphism.
\end{definition}

Proposition~\ref{prop free quan} tells us that when viewing bundles from an algebraic perpespective, the freeness condition is equivalent to the Hopf-Galois extension property. Hence, the Hopf-Galois extension condition is a necessary condition to ensure a bundle is principal. Not all information about a topological space is encoded in a coordinate algebra, so to make a fuller reflection of the richness of the classical notion of a principal bundle we need to require  conditions additional to the Hopf-Galois property. 

\begin{definition}\label{def.princ} \rm
Let $H$ be a Hopf algebra with bijective antipode and let $A$ be a right $H$-comodule algebra with coaction $\roA: A\to A\ot H$. Let $B$ denote the coinvariant subalgebra of $A$. We say that $A$ is a {\em principal $H$-comodule algebra} if:

(a)  $B \subseteq A$ is a Hopf-Galois extension; 

(b)  the multiplication map
$
B\ot A \to A$, $b\ot a\mapsto ba$, 
splits as a left $B$-module and right $H$-comodule map (the equivariant projectivity condition).
\end{definition}
As indicated already in \cite{Sch:pri}, \cite{BrzMaj:gau} or \cite{Haj:str}, principal comodule algebras should be understood as principal bundles in noncommutative geometry. In particular, if $H$ is a $C^*$-algebra of functions on a quantum group \cite{Wor:com}, then the existence of the Haar measure together with the results of \cite{Sch:pri} mean that the freeness of the coaction implies its principality. 

The following characterisation of principal comodule algebras \cite{DabGro:str}, \cite{BrzHaj:Che} gives an effective method for proving the principality of coaction.

\begin{proposition}\label{prop str con}
A right $H$-comodule algebra $A$ with coaction $\roA: A\to A\ot H$  is principal if and only if it admits a {\em strong connection form}, that is if there exists a map
$
\omega:H\longrightarrow A\otimes A,
$
such that
\begin{subequations}
\label{strong}
\begin{equation}
\omega(1)=1\otimes 1, \label{strong1} 
\end{equation}
\begin{equation}
\mu\circ \omega = \eta \circ \varepsilon, \label{strong2}
\end{equation}
\begin{equation}
 (\omega\otimes\id)\circ\Delta  = (\id\otimes \varrho)\circ \omega , \label{strong3}
\end{equation}
\begin{equation}
(S\otimes \omega)\circ\Delta = (\sigma\otimes \id)\circ (\varrho\otimes \id)\circ \omega .  \label{strong4}
\end{equation}
\end{subequations}
Here $\mu: A\ot A\to A$ denotes the multiplication map, $\eta: \CC\to A$ is the unit map, $\Delta: H\to H\ot H$ is the comultiplication, $\eps: H\to\CC$ counit and $S: H\to H$ the (bijective) antipode of the Hopf algebra $H$, and $\sigma : A\ot H \to H\ot A$ is the flip.
\end{proposition}
\begin{proof}
 If a strong connection form $\omega$ exists, then the inverse of the canonical map $\can$ (see Definition~\ref{def.H=G} ) is the composite
$$
\xymatrix{
A\ot H \ar[rr]^-{\id \ot \omega} && A\ot A\ot A \ar[rr]^{\mu \ot\id} && A\ot A \ar[r] & A\ot_B A},
$$
while the splitting of the multiplication map (see Definition~\ref{def.princ}~(b)) is given by
$$
\xymatrix{
A \ar[r]^-{\roA} & A\ot H \ar[rr]^{\id \ot\omega} && A\ot A\ot A \ar[rr]^-{\mu \ot\id} && B\ot  A}.
$$

Conversely, if $B\subseteq A$ is a principal comodule algebra, then $\omega$ is the composite
$$
\xymatrix{
H \ar[r]^-{\eta\ot\id} & A\ot H \ar[r]^-{\can^{-1}} & A\ot_B A \ar[r]^-{\id\ot s}& A\ot_B B\ot  A\ar[r]^-{\cong} & A\ot A},
$$
where $s$ is the left $B$-linear right $H$-colinear splitting of the multiplication $B\ot A\to A$.
\end{proof}

\begin{example}\label{ex cleft} \rm Let $A$ be a right $H$-comodule algebra. The space of $\CC$-linear maps $\rhom{}HA$ is an algebra with the {\em convolution product} 
$$
f\ot g \mapsto \mu\circ (f\ot g) \circ \Delta
$$
and unit $\eta\circ \eps$. $A$ is said to be {\em cleft} if there exists a right $H$-colinear map $j: H\to A$ that has an inverse in the convolution algebra $\rhom{}HA$ and is normalised so that $j(1)=1$. Writing $j^{-1}$ for the convolution inverse of $j$, one easily observes that
$$
\omega: H\to A\ot A, \qquad h\mapsto (j^{-1}\ot j)(\Delta(h)),
$$ 
is a strong connection form. Hence a cleft comodule algebra is a an example of a principal comodule algebra. The map $j$ is called a {\em cleaving map} or a {\em normalised total integral}.

In particular, if $j:H\to A$ is an $H$-colinear algebra map, then it is automatically convolution invertible (as $j^{-1} = j\circ S$) and normalised. A comodule algebra $A$ admitting such a map is termed a {\em trivial} principal comodule algebra.
\end{example}

\begin{example} \rm
Let $H$ be a Hopf algebra of the compact quantum group. By the Woronowicz theorem \cite{Wor:com}, $H$ admits an invariant Haar measure, i.e.\ a linear map $\Lambda: H \to \CC$ such that, for all $h\in H$,
$$
\sum h\sw 1 \Lambda (h\sw 2) = \eps(h), \qquad \Lambda(1) =1. 
$$
where $\Delta(h) = \sum h\sw 1 \ot h\sw 2$ is the Sweedler notation for the comultiplication. Next, assume  that the lifted canonical map:
\begin{equation}\label{bar can}
\overline{\can}: A\ot  A\to A\ot H, \qquad a\ot a'\mapsto a\varrho(a'),
\end{equation}
is surjective, and write 
$$
\ell: H \to A\ot A, \qquad \ell(h) = \sum \ell(h)\su 1 \ot \ell(h)\su 2,
$$
for the $\CC$-linear map such that $\overline{\can}(\ell(h)) =1\ot h$, for all $h\in H$. Then, by the Schneider theorem \cite{Sch:pri}, $A$ is a principal $H$-comodule algebra. Explicitly, a strong connection form is
$$
\omega(h) = \sum \Lambda\left(h\sw 1 \ell\left(h\sw 2\right)\su 1\sw 1\right)\Lambda\left(\ell\left(h\sw 2\right)\su 2\sw 1S\left(h\sw 3\right)\right)\ell(h\sw 2)\su 1\sw 0\ot \ell(h\sw 2)\su 2\sw 0,
$$
where the coaction is denoted by the Sweedler notation $\roA(a) = \sum a\sw 0\ot a\sw 1$; see \cite{BegBrz:exp}.
\end{example}
Having described non-commutative principal bundles, we can look at the associated vector bundles. First we look at the classical case and try to understand it purely algebraically. 
Start with a vector bundle $(E,\pi_E,Y,V)$ associated to a principal $G$-bundle $X$. Since $V$ is a vector representation space of $G$, also the set  $\rhom G X V$ is a vector space. Consequently $\Gamma(E)$ is a vector space. Furthermore, $\rhom G X V$ is a left module of $B =\cO(Y)$ with the action
$
(bf)(x) = b(\pi_E(x))f(x).
$
To understand better the way in which $B$-module  $\Gamma(E)$ is associated to the principal comodule algebra $\cO(X)$ we recall the notion of the cotensor product.
\begin{definition}  \rm Given a Hopf algebra $H$, right $H$-comodule $A$ with coaction $\roA$ and left $H$-comodule $V$ with coaction ${}^V\! \varrho$, the {\em cotensor product} is defined as an equaliser:
$$\xymatrix{ A\Box_HV \ar[r]& A\ot V
\ar@<0.5ex>[rr]^{\roA\ut \id\quad}\ar@<-0.5ex>[rr]_{\id \ut {}^V\! \varrho\quad}
& & A\ot 
 H\ot  V }.$$
 \end{definition}
 
If $A$ is an $H$-comodule  algebra, and $B=A^{coH}$, then $A\Box_HV$ is a left $B$-module with the action
$
b(a\Box v) = ba\Box v.
$
In particular, in the case of a principal $G$-bundle $X$ over $Y=X/G$, for any left $\cO(G)$-comodule $V$ the cotensor product $\cO(X)\Box_{\cO(G)}V$ is a left $\cO(Y)$-module.

Assume that $V$ is finite dimensional. Then the dual vector space $V$ is a left comodule of $\cO(G)$ with the coaction
$
{}^{V}\!\varrho : v\mapsto \sum v\sw{-1}\ot v\sw{0}
$
(summation implicit) determined by
$
\sum v\sw{-1}(g) v\sw{0} = g^{-1}\triangleright v.
$

\begin{proposition} \label{prop clas sec} The left $\cO(Y)$-module of sections $\Gamma(E)$ is isomorphic to the left $\cO(Y)$-module $\cO(X)\Box_{\cO(G)}V$.
\end{proposition}
\begin{proof}
First identify $\Gamma(E)$ with $\rhom G X V$. Let $\{v_i\in V^*,\; v^i\in V\}$ be a (finite) dual basis. Take $f\in \rhom G X V$, and define $\theta: \rhom G X V \to \cO(X)\Box_{\cO(G)}V$ by
$
\theta(f) = \sum_i v_i\circ f\ot v^i.
$

In the converse direction, define a left $\cO(Y)$-module map
$$
\theta^{-1}: \cO(X)\Box_{\cO(G)}V\to \rhom G X V, \qquad a\Box v\mapsto a(-)v.
$$
One easily checks that the constructed map are mutual inverses. 
\end{proof}

Moving away from commutative algebras of functions on topological spaces one uses Proposition~\ref{prop clas sec} as the motivation for the following definition.
\begin{definition}\label{def ass mod}\rm 
Let $A$ be a principal $H$-comodule algebra. Set $B=A^{coH}$ and let $V$ be a left $H$-comodule. The left $B$-module $\Gamma = A\Box_H V$ is called a {\em module    associated to the principal comodule algebra $A$}.
\end{definition}

$\Gamma$ is a projective left $B$-module, and if  $V$ is a finite dimensional vector space, then  $\Gamma$ is a finitely generated projective left $B$-module. In this case it has the meaning of a module of sections over a non-commutative vector bundle. Furthermore, its class gives an element in the $K_0$-group of $B$. If $A$ is a cleft principal comodule algebra, then every associated module is free, since $A\cong B\ot H$ as a left $B$-module and right $H$-comodule, so that
$$
\Gamma = A\Box_H V \cong (B\ot H)\Box_H V \cong B\ot (H\Box_H V) \cong B\ot  V.
$$

\section{Weighted circle actions on prolonged spheres.}

In this section we recall the definitions of algebras we study in the sequel. 

\subsection{Circle actions and $\ZZ$-gradings.}\label{section.zz}

The coordinate algebra of the circle or the group $U(1)$, $\cO(S^1) = \cO(U(1))$ can be identified with the $*$-algebra $\CC[u,u^*]$ of Laurent polynomials in a unitary variable $u$ (unitary means $u^{-1} = u^*$). As a Hopf $*$-algebra $\CC[u,u^*]$, is generated by the grouplike element $u$, i.e.\
$$
\Delta(u) = u\ot u, \qquad \eps(u) = 1, \qquad S(u) = u^*,
$$
and thus it can be understood as the group algebra $\CC\ZZ$. As a consequence of this interpretation of $\CC[u,u^*]$, an algebra $A$ is a $\CC[u,u^*]$-comodule algebra if and only if $A$ is a $\ZZ$-graded algebra,
$$
A = \bigoplus_{n\in \ZZ} A_n, \qquad A_n:= \{a\in A\; |\; \roA(a) = a\ot u^n\}, \qquad A_mA_n\subseteq A_{m+n}.
$$
$A_0$ is the coinvariant subalgebra of $A$. Since $\CC[u,u^*]$ is spanned by grouplike elements, any convolution invertible map $j: \CC[u,u^*]\to A$ must assign a  unit (invertible element) of $A$ to $u^n$. Furthermore,  colinear maps are simply the $\ZZ$-degree preserving maps, where $\deg(u) = 1$. Put together, convolution invertible colinear maps $j: \CC[u,u^*]\to A$ are in one-to-one correspondence with sequences 
$$
(a_n\; :\; n\in \ZZ, \;a_n \mbox{ is a unit in $A$}, \; \deg(a_n) = n).
$$

\subsection{The $\cO(\Sigma_q^{2n+1})$ and $\cO(\RR \PP_q(l_0,...,l_n))$ coordinate algebras.}

Let  $q$ be a real number, $0< q<1$. 
The coordinate algebra $\cO(S^{2n}_q)$ of  the even-dimensional quantum sphere is the unital complex $*$-algebra with generators $z_0,z_1,\ldots ,z_n$, subject to the following relations:
\begin{subequations}
\begin{equation}
 z_iz_j = qz_jz_i \quad \mbox{for $i<j$}, \qquad z_iz^*_j = qz_j^*z_i \quad \mbox{for $i\neq j$},
\end{equation}
\begin{equation}
z_iz_i^* = z_i^*z_i + (q^{-2}-1)\sum_{m=i+1}^n z_mz_m^*, \qquad \sum_{m=0}^n z_mz_m^*=1, \qquad z_n^* =z_n.
\end{equation}
\end{subequations}
$\cO(S^{2n}_q)$ is a $\ZZ_2$-graded algebra with $\deg(z_i) =1$ and so is $\CC[u,u^*]$ (with $\deg(u) =1$). In other words, $\cO(S^{2n}_q)$ is a right $\CC\ZZ_2$-comodule algebra and $\CC[u,u^*]$ is a left $\CC\ZZ_2$-comodule algebra, hence one can consider the cotensor product algebra $\cO(\Sigma_q^{2n+1}):=\cO(S^{2n}_q)\Box_{\CC\ZZ_2}\CC[u,u^*]$. It was shown in \cite{BrzZie:pri} that, as a unital $*$-algebra, $\cO(\Sigma_q^{2n+1})$ has generators $\zeta_0,...,\zeta_n$ and a central unitary $\xi$ which are related in the following way:
\begin{subequations}\label{sph}
\begin{equation}\label{sph.1}
\zeta_i \zeta_j = q \zeta_j \zeta_i \quad \mbox{for $i<j$}, \qquad \zeta_i \zeta^*_j = q \zeta_j^* \zeta_i \quad \mbox{for $i\neq j$}, 
\end{equation}
\begin{equation}\label{sph.2}
\zeta_ i \zeta_i^* = \zeta_i^*\zeta_i + (q^{-2}-1)\sum_{m=i+1}^n \zeta_m \zeta_m^*, \qquad \sum_{m=0}^n \zeta_m \zeta_m^*=1, \qquad \zeta_n^*=\zeta_n \xi .
\end{equation}
\end{subequations}

For any choice of $n+1$ pairwise coprime numbers $l_0,...,l_n$ one can define the coaction of the Hopf algebra $\cO(U(1))=\CC[u,u^*]$ on $\cO(\Sigma_q^{2n+1})$ as 
\begin{equation}
\varrho_{l_0,...,l_n}: \cO(\Sigma_q^{2n+1}) \rightarrow \cO(\Sigma_q^{2n+1}) \otimes \CC[u,u^*], \qquad \zeta_i \mapsto \zeta_i \otimes u^{l_i}, \qquad \xi \mapsto \xi \otimes u^{-2l_n},
\end{equation}
for $i=0,1,...,n$. This coaction is then extended to the whole of $\cO(\Sigma_q^{2n+1})$ so that $\cO(\Sigma_q^{2n+1})$ is a right $\CC[u,u^*]$-comodule algebra.

The algebra of coordinate functions on the quantum real weighted projective space is now defined as the subalgebra of $\cO(\Sigma_q^{2n+1})$ containing all coinvariant elements, i.e.,
$$
\cO(\RR \PP_q(l_0,...,l_n))=\cO(\Sigma_q^{2n+1})^{\cO(U(1))}:=\{x \in \cO(\Sigma_q^{2n+1}):\varrho_{l_0,...,l_n}(x)=x \otimes 1\}.
$$

\subsection{The 2D quantum real projective space $\cO(\RR \PP_q(k,l)) \subset \cO(\Sigma_q^3)$.}

In this paper we consider two-dimensional quantum real weighted projective spaces, i.e.\ the algebras obtained from the coordinate algebra $\cO(\Sigma_q^3)$ which
 is generated by $\zeta_0, \zeta_1$ and central unitary $\xi$ such that 
\begin{subequations}\label{zeta rel} 
\begin{equation}
\zeta_0 \zeta_1 = q \zeta_1 \zeta_0, \quad \zeta_0 \zeta^*_1 = q \zeta_1^* \zeta_0, \label{zeta rel a}
\end{equation}
\begin{equation}
\zeta_ 0\zeta_0^* = \zeta_0^*\zeta_0 + (q^{-2}-1) \zeta^2_1 \xi, \qquad \zeta_0 \zeta_0^*+\zeta_1^2 \xi=1, \qquad \zeta_1^*=\zeta_1 \xi.\label{zeta rel b}
\end{equation}
\end{subequations}
The linear basis of $\cO(\Sigma_q^3)$ is 
\begin{equation}\label{basis.sigma}
\{\zeta_0^r \zeta_1^s \xi^t,\;  \zeta_0^{*r} \zeta_1^s \xi^t, \; |\; r,s,\in\NN, \; t\in \ZZ\}.
\end{equation}
For a pair $k,l$ of coprime integers, the coaction $\varrho_{k,l}$ is given on generators by
\begin{equation}
\zeta_0 \mapsto \zeta \otimes u^k, \qquad \zeta_1 \mapsto \zeta_1 \otimes u^l, \qquad \xi \mapsto \xi \otimes u^{-2l},
\end{equation}
and extended to the whole of $\cO(\Sigma_q^3)$ so that the coaction is a $*$-algebra map. We denote the comodule algebra $\cO(\Sigma_q^3)$ with coaction $\varrho_{k,l}$ by $\cO(\Sigma_q^3(k,l))$.

It turns out that the two dimensional quantum real projective spaces split into two cases depending not wholly on the parameter $k$, but instead whether $k$ is either even or odd, and hence only cases $k=1$ and $k=2$ need be considered \cite{Brz:sei}. We describe these cases presently.

\subsubsection{~The odd or negative case.}

For $k=1$, $\cO(\RR \PP_q^2(l;-))$ is a polynomial  $*$-algebra generated  by $a$, $b$, $c_-$  which satisfy the relations:
 \begin{subequations}
\begin{equation}
a=a^*, \qquad ab=q^{-2l}ba, \qquad ac_{-}=q^{-4l}c_-a, \qquad b^2=q^{3l}ac_-, \qquad bc_-=q^{-2l}c_-b, 
\end{equation}
\begin{equation}
 bb^*=q^{2l}a \prod_{m=0}^{l-1} (1-q^{2m}a), \qquad b^*b=a \prod_{m=1}^{l} (1-q^{-2m}a), 
 \end{equation}
\begin{equation}
 b^*c=q^{-l} \prod_{m=1}^{l} (1-q^{-2m}a)b, \qquad c_-b^*=q^{l}b \prod_{m=0}^{l-1} (1-q^{2m}a), 
 \end{equation}
\begin{equation}
 c_-c_-^*= \prod_{m=0}^{2l-1} (1-q^{2m}a), \qquad c_-^*c_-= \prod_{m=1}^{2l} (1-q^{-2m}a).
\end{equation}
\end{subequations}
The embedding of generators of $\cO(\RR \PP_q^2(l;-))$ into $\cO(\Sigma_q^3)$ or the isomorphism of $\cO(\RR \PP_q^2(l;-))$ with the coinvariants of  $\cO(\Sigma_q^3(1,l))$ is provided by 
\begin{equation}\label{embed-}
a\mapsto \zeta_1^2 \xi, \qquad b\mapsto \zeta_0^l \zeta_1 \xi, \qquad c_-\mapsto \zeta_0^{2l} \xi.
\end{equation}

Up to equivalence $\cO(\RR\PP_q^2(l;-))$  has the following irreducible $*$-representations. There is a family of one-dimensional representations of labelled by $\theta\in [0,1)$ and given by
\begin{equation}\label{rep.o}
\pi_\theta(a) = 0, \qquad \pi_\theta(b) = 0, \qquad  \pi_\theta(c_-) = e^{2\pi i\theta}.
\end{equation}
All other representations are infinite dimensional, labelled by $r=1,\ldots , l$, and given by
\begin{subequations}
\begin{equation}\label{reps.o.1}
\pi_r(a) e_n^r = q^{2(ln+r)} e_n^r, \quad \pi_r(b) e_n^r =  q^{ln+r}\prod_{m=1}^{l}\left(1 - q^{2(ln+r-m)}\right)^{1/2}e_{n-1}^r, \quad \pi_r(b) e_0^r =0,
\end{equation}
\begin{equation} \label{reps.o}
\pi_r(c_-) e_n^r =  \prod_{m=1}^{2l}\left(1 - q^{2(ln+r-m)}\right)^{1/2}e_{n-2}^r, \qquad \pi_r(c_-) e_0^r = \pi_r(c_-) e_1^r= 0, 
\end{equation}
\end{subequations}
where $e_n^r$, $n\in \NN$, is an orthonormal basis for the representation space $\Hh_r \cong l^2(\NN)$.

The $C^*$-algebra of continuous functions on $\RR\PP_q^2(l;-)$, obtained as the completion of these bounded representations, can be identified with the pullback of $l$-copies of the quantum real projective plane $\RR\PP_q^2$ introduced in \cite{HajMat:rea}.

\subsubsection{~The even or positive case.}

 For $k=2$  and hence $l$ odd, $\cO(\RR \PP_q^2(l;+))$ is a polynomial $*$-algebra generated  by $a$, $c_+$  which satisfy the relations:
\begin{subequations}
\begin{equation}
a^*=a, \qquad ac_+=q^{-2l}c_+a, 
\end{equation}
\begin{equation}
 c_+c_+^*= \prod_{m=0}^{l-1}(1-q^{2m}a), \qquad c_+^* c_+= \prod_{m=1}^l (1-q^{-2m}a).
\end{equation}
\end{subequations}
The embedding of generators of $\cO(\RR \PP_q^2(l;+))$ into $\cO(\Sigma_q^3)$ or the isomorphism of $\cO(\RR \PP_q^2(l;+))$ with the coinvariants of  $\cO(\Sigma_q^3(2,l))$ is provided by 
\begin{equation}\label{embed+}
a\mapsto \zeta_1^2 \xi, \qquad c_+\mapsto \zeta_0^l \xi.
\end{equation}

Similarly to the odd $k$ case, there is a family of one-dimensional representations of $\cO(\RR\PP_q^2(l;+))$ labelled by $\theta\in [0,1)$ and given by
\begin{equation}\label{rep.e}
\pi_\theta(a) = 0, \qquad \pi_\theta (c_+) = e^{2\pi i\theta}.
\end{equation}
All other representations are infinite dimensional, labelled by $r=1,\ldots , l$, and given by
\begin{equation}\label{reps.e}
\pi_r(a) e_n^r = q^{2(ln+r)} e_n^r, \quad \pi_r(c_+) e_n^r =  \prod_{m=1}^{l}\left(1 - q^{2(ln+r-m)}\right)^{1/2}e_{n-1}^r, \quad \pi_r(c_+) e_0^r =0,
\end{equation}
where $e_n^r$, $n\in \NN$ is an orthonormal basis for the representation space $\Hh_r\cong l^2(\NN)$.

The $C^*$-algebra $C(\RR\PP_q^2(l;+))$ of continuous functions on $\RR\PP_q^2(l;+)$, obtained as the completion of these bounded representations, can be identified with the pullback of $l$-copies of the quantum disk $D_q$ introduced in \cite{KliLes:two}. Furthermore, $C(\RR\PP_q^2(l;+))$ can also be understood as the quantum double suspension of $l$ points in the sense of \cite[Definition~6.1]{HonSzy:sph}.

 \section{Quantum real weighted  projective spaces and quantum principal bundles.}
 
 The general aim of this paper is to construct quantum principal bundles with base spaces given by $\cO(\RR \PP_q^2(l;\pm))$ and fibre structures given by the circle  Hopf algebra $\cO(S^1) \cong \CC[u,u^*]$. The question arises as to which quantum space (i.e.\ a $\CC[u,u^*]$-comodule algebra with coinvariants isomorphic to $\cO(\RR \PP_q^2(l;\pm))$) we should consider as the total space within this construction.  We look first at the  coactions of $\CC[u,u^*]$ on $\cO(\Sigma_q^3)$ that define $\cO(\RR \PP_q(k,l))$, i.e.\ at  the comodule algebras $\cO(\Sigma_q^3(k,l))$.

 \subsection{The (non-)principality of $\cO(\Sigma_q^3(k,l))$.}

\begin{theorem} $A=\cO(\Sigma_q^3(k,l))$ is a principal comodule algebra if and only if  $(k,l) = (1,1)$. \label{HG}
\end{theorem}

\begin{proof}
As explained in \cite{BrzZie:pri} $\cO(\Sigma_q^3(1,1))$ is a prolongation of the $\CC\ZZ_2$-comodule algebra $\cO(S_q^2)$. The latter is a principal comodule algebra (over the quantum real projective plane $\cO(\RR\PP_q^2)$ \cite{HajMat:rea}) and since a prolongation of  a principal comodule algebra is a principal comodule algebra \cite[Remark~3.11]{Sch:pri}, the coaction $\varrho_{1,1}$ is principal as stated.

In the converse direction, we aim to show that the canonical map is not an isomorphism by showing that the image does not contain $1 \otimes u$, i.e. it cannot be surjective since we know $1 \otimes u$ is in the codomain. We begin by identifying a basis for the algebra $\cO(\Sigma_q^3) \otimes \cO(\Sigma_q^3)$; observing the relations (\ref{zeta rel a}) and (\ref{zeta rel b}) it is clear that a basis for $\cO(\Sigma_q^3(k,l))$ is given by linear combinations of elements of the form,
$$
b_1=b_1(p_1,p_2,p_3)=\zeta_0^{p_1}\zeta_1^{p_2} \xi^{p_3}, \qquad 
b_2=b_2(\bar{p_1},\bar{p_2},\bar{p_3})=\zeta_0^{\bar{p_1}}\zeta_1^{\bar{p_2}} \xi^{\bar{* p_3}}, 
$$
$$
b_3=b_3(q_1,q_2,q_3)=\zeta_0^{* q_1}\zeta_1^{q_2} \xi^{q_3}, \qquad
b_4=b_4(\bar{q_1},\bar{q_2},\bar{q_3})=\zeta_0^{* \bar{q_1}}\zeta_1^{\bar{q_2}} \xi^{*\bar{q_3}},
$$
noting that all powers are non-negative. Hence a basis for $\cO(\Sigma_q^3) \otimes \cO(\Sigma_q^3)$ is given by linear combinations of elements of the form
$b_i \otimes b_j$, {where} $ i,j \in \{1,2,3,4\}$.
Applying the canonial map gives 
\begin{equation}
\can(b_i \otimes b_j)=b_i \varrho(b_j)=b_ib_j \otimes u^{deg(b_j)}, \qquad \text{where} \ \ \  i,j \in \{1,2,3,4\}, 
\end{equation}
where $\varrho$ means $\varrho_{k,l}$ for simplicity of notation. The next stage is to construct all possible elements in $\cO(\Sigma_q^3) \otimes \cO(\Sigma_q^3)$ which map to $1 \otimes u$. To obtain the identity in the first leg we must use one of the following relations:
\begin{subequations}\label{key.rel}
\begin{equation}
\zeta_0^m {\zeta_0^*}^n=
\begin{cases}
\prod_{p=0}^{m-1} \mathcal(1-q^{2p}\zeta_1^2 \xi) & \mbox{when \ $m=n$,}\\
\zeta_0^{m-n} \prod_{p=0}^{n-1} \mathcal(1-q^{2p}\zeta_1^2 \xi) & \mbox{when \ $m > n$,}\\
\prod_{p=0}^{m-1} \mathcal(1-q^{2p}\zeta_1^2 {\xi}) {\zeta_0^*}^{n-m}  & \mbox{when\ $n > m$,}
\end{cases}
\end{equation}
\begin{equation}
{\zeta_0^*}^n \zeta_0^m=
\begin{cases}
\prod_{p=1}^{m} \mathcal(1-q^{-2p}\zeta_1^2 {\xi}) &\mbox{when\ $m=n$,}\\
{\zeta_0^*}^{n-m} \prod_{p=1}^{m} \mathcal(1-q^{-2p}\zeta_1^2 {\xi}) & \mbox{when \ $n > m$,}\\
\prod_{p=1}^{n} \mathcal(1-q^{-2p}\zeta_1^2 {\xi}) \zeta^{m-n}  & \mbox{when \ $n < m$.}
\end{cases}
\end{equation}
\end{subequations}
or
\begin{equation}
\xi \xi^*=\xi^* \xi=1 \notag
\end{equation}
We see that to obtain identity in the first leg we require the powers of $\zeta_0$ and ${\zeta_0^*}$ to be equal. We now construct all possible elements of the domain which map to $1 \otimes u$ after applying the canonical map.

\underline{Case 1}: use the first relation to obtain $\zeta_0^m \zeta_0^{*m}$ ($m>0$); this can be done in fours ways. First, using $b_1 \varrho(b_3)$, $b_1 \varrho(b_4)$, $b_2 \varrho(b_3)$ and $b_2 \varrho(b_4)$. Now,
$$
b_1 \varrho(b_3) \sim \zeta_0^{p_1}\zeta_0^{* q_1}\zeta_1^{p_2+q_2}\xi^{p_3+q_3} \otimes u^{-kq_1+lq_2-2lq_3} \Longrightarrow p_1=q_1=m, p_2=q_2=0, p_3=q_3=0,
$$
and 
$$
-kq_1+lq_2-2lq_3=1 \Longrightarrow -mk=1,
$$ 
hence no possible terms. A similar calculation for the three other cases shows that $1 \otimes u$ cannot be obtained as an element of the image of the canonical map in this case.

\underline{Case 2}: use the second relation to obtain $\zeta_0^{*n} \zeta_0^{n}$ ($n>0$); this can be done in four ways $b_3 \varrho(b_1)$, $b_3 \varrho(b_2)$, $b_3 \varrho(b_2)$ and $b_4 \varrho(b_2)$. Now,
$$
b_3 \varrho(b_1) \sim \zeta_0^{*q_1}\zeta_0^{p_1}\zeta_1^{p_2+q_2}\xi^{p_3+q_3} \otimes u^{k p_1+l p_2-2lp_3} \Longrightarrow p_1=q_1=n, p_2=q_2=0, p_3=q_3=0
$$ 
and 
$$
nk=1 \Longrightarrow n=1 \ \text{and} \ k=1.
$$ 
Note that $k=1$ is not a problem provided $l$ is not equal to $1$. This is reviewed at the next stage of the proof. The same conclusion is reached in all four cases.

In all possibilities $\zeta_0^{* n} \zeta_0^n$ appears only when $n=1$, in which case the relation simplifies to $\zeta_0^{*} \zeta_0=1-q^{-2}\zeta_1^2 \xi$, so the next stage involves constructing elements in the domain which map to $\zeta_1^2 \xi$. There are eight possibilities altogether to be checked: $b_1 \varrho(b_1)$, $b_1 \varrho(b_2)$, $b_1 \varrho(b_3)$, $b_1 \varrho(b_4)$, $b_3 \varrho(b_1)$, $b_3 \varrho(b_2)$, $b_3 \varrho(b_3)$ and $b_3 \varrho(b_4)$.
The first case gives:
$$
b_1 \varrho(b_1) \sim \zeta_0^{2p_1} \zeta_1^{2p_2} \xi^{2p_3} \otimes u^{kp_1+lp_2-2lp_3} \Longrightarrow \ 2p_1=0, \ 2p_2=2, \ 2p_3=1,
$$ 
and 
$$
kp_1+lp_2-2lp_3=1 \Longrightarrow p_1=0, p_2=1, p_3 \ \mbox{has no possible values and} \  l=1. 
$$ 
Hence $1 \otimes u$ cannot be obtained as an element in the image in this case. Similar calculations for the remaining possibilities show that either $1 \otimes u$ is not in the image of the canonical map, or that if $1 \otimes u$ is in the image then $k=l=1$.

\underline{Case 3}: finally, it seems possible that $1 \otimes u$, using the third relation, could be in the image of the canonical map. All possible elements in the domain which could potentially map to this element are constructed and investigated. There are eight possibilities: $b_1 \varrho(b_2)$, $b_1 \varrho(b_4)$, $b_2 \varrho(b_1)$, $b_2 \varrho(b_3)$, $b_3 \varrho(b_2)$, $b_3 \varrho(b_4)$, $b_4 \varrho(b_1)$ and $b_4 \varrho(b_3)$. The first possibility comes out as
$$
b_1 \varrho(b_2) \sim \zeta_0^{p_1+\bar{p}_1}\zeta_1^{p_2+\bar{p}_ 2} \xi^{p_3} \xi^{* \bar{p}_3} \otimes u^{k \bar{p}_1+l\bar{p}_2+2l \bar{p}_3} \Longrightarrow p_1=\bar{p}_1=0, p_2=\bar{p_2}=0, p_3=\bar{p_3}=1.
$$
Also
$$
k \bar{p_1}+l \bar{p}_2+2l \bar{p}_3=1 \Longrightarrow 2l=1,
$$ 
which implies there are no terms. The same conclusion can be reached for the remaining relations.

This concludes that $1 \otimes u$, which is contained in $\cO(\Sigma_q^3) \otimes \otimes \CC[u,u^*]$, is not in the image of the canonical map, proving that this map is not surjective and ultimately not an isomorphism when $k$ and $l$ are both not simultaneously equal to $1$, completing the proof that $\cO(\Sigma_q^3(k,l))$ is not a principal comodule algebra in this case.
\end{proof}

Theorem~\ref{HG} tells us that if we use $\cO(\Sigma^3_q(k,l))$ as our total space, then we are forced to put $(k,l)=(1,1)$ to ensure that the required Hopf-Galois condition does not fail. A consequence of this would be the generators $\zeta_0$ and $\zeta_1$ would have $\ZZ$-degree $1$. This suggests that the comodule algebra $\cO(\Sigma^3_q(k,l))$ is to restrictive as there is no freedom with the weights $k$ or $l$, and that we should in fact consider a subalgebra of $\cO(\Sigma_q^3)$ which admits a $\cO(S^1)$-coaction that would offer some choice. Theorem~\ref{HG} indicates that the desired subalgebra should have generators with grades $1$ to ensure the Hopf-Galois condition is satisfied. This process in similar to that followed in  \cite{BrzFai:tea}, where the bundles over the quantum teardrops $\WW \PP_q(1,l)$ have the total spaces provided by the quantum lens spaces and structure groups provided by the circle group $U(1)$. We follow a similar approach in the sense that we view $\cO(\Sigma^3_q(k,l))$ as a right $H$-comodule algebra, where $H$ is the Hopf algebra of a suitable cyclic group.  

\subsection{The negative case $\cO(\RR \PP_q^2(l;-))$.}

\subsubsection{~The principal $\cO(U(1))$-comodule algebra over $\cO(\RR \PP_q^2(l;-))$.}

Take the group Hopf $*$-algebra $H=\CC\ZZ_l$ which is generated by unitary grouplike element $w$ and satisfies the relation $w^l=1$. The algebra $\cO(\Sigma_q^3)$ is a right $\CC\ZZ_l$-comodule $*$-algebra with coaction
\begin{equation}
\cO(\Sigma_q^3) \rightarrow \cO(\Sigma_q^3) \otimes \CC\ZZ_l, \qquad \zeta_0 \mapsto \zeta_0 \otimes w, \ \ \zeta_1 \mapsto \zeta_1 \otimes 1, \ \ \xi \mapsto \xi \otimes 1.
\end{equation} 
Note that the $\ZZ_l$-degree of the  generator $\xi$ is determined by the degree of $\zeta_1$: the relation $\zeta_1^*=\zeta_1 \xi$ and that the coaction must to compatible with all relations imply that $\deg(\zeta_1^*)=\deg(\zeta_1)+\deg(\xi)$. Since $\zeta_1$ has degree zero, $\xi$ must also have degree zero.

The next stage of the process is to find the coinvariant elements of $\cO(\Sigma_q^3)$ given the coaction defined above. 
\begin{proposition} \label{fix pts} The fixed point subalgebra of the above coaction is isomorphic to the algebra $\cO(\Sigma^3_q(l;-))$, generated by $x$, $y$ and $z$ subject to the following relations 
\begin{equation}
y^* = yz, \qquad xy=q^l yx, \qquad xx^*=\prod_{p=0}^{l-1}(1-q^{2p}y^2z), \qquad x^*x=\prod_{p=1}^{l}(1-q^{-2p}y^2z), \label{R1}
\end{equation}
and $z$ is central unitary. The embedding of $\cO(\Sigma^3_q(l;-))$ into $\cO(\Sigma_q^3)$ is given by $x\mapsto \zeta_0^l$, $y\mapsto \zeta_1$ and $z\mapsto \xi$.
\end{proposition}
\begin{proof}
Clearly $\zeta_1$, $\xi$,  $\zeta_0^{l}$ and  $\zeta_0^{*l}$ are coinvariant elements of $\cO(\Sigma_q^3)$. Apply the coaction to the basis \eqref{basis.sigma} to obtain
$$
\zeta_0^r \zeta_1^s \xi^t \mapsto \zeta_0^r \zeta_1^s \xi^t\ot w^r, \qquad \zeta_0^{*r} \zeta_1^s \xi^t \mapsto \zeta_0^{*r} \zeta_1^s \xi^t\ot w^{-r} .
$$ 
These elements are coinvariant, provided $r=r'l$. Hence every coinvariant element is a polynomial in  $\zeta_1$, $\xi$,  $\zeta_0^{l}$ and  $\zeta_0^{*l}$. Relations \eqref{R1} are now easily derived from \eqref{zeta rel} and \eqref{key.rel}.
\end{proof}

The algebra $\cO(\Sigma^3_q(l;-))$ is a right $\cO(U(1))$-comodule coalgebra with coaction defined as
\begin{equation}
\varphi:\cO(\Sigma^3_q(l;-)) \rightarrow \cO(\Sigma^3_q(l;-)) \otimes \cO(U(1)), \qquad x \mapsto x \otimes u, \qquad y \mapsto y \otimes u, \qquad z \mapsto z \otimes u^{-2}.
\end{equation}
Note in passing that the second and third relations in (\ref{R1}) tell us that the grade of $z$ must be double the grade of $y^*$ since $xx^*$ and $x^*x$ have degree zero, and so 
$$
\deg(y^2 z)= \deg(y^2)+\deg(z)=2\deg(y)+\deg(z)=0  \Longrightarrow  \deg(z)=-2\deg(y)=2\deg(y^*). 
$$

\begin{proposition}  The algebra $\cO(\Sigma^3_q(l;-))^{co\cO(U(1))}$ of invariant elements under the coaction $\varphi$ is isomorphic to the $\cO((\RR \PP_q(l;-))$. 
\end{proposition}

\begin{proof}
We aim to show that the $*$-subalgebra of $\cO(\Sigma_q^3(l;-))$ of elements which are invariant under the coaction is generated by  $x^2z$, $xyz$ and $y^2z$.  The isomorphism of $\cO(\Sigma^3_q(l;-))^{co\cO(U(1))}$ with $\cO((\RR \PP_q(l;-))$  is then obtained by using the embedding of $\cO(\Sigma_q^3(l;-))$ in $\cO(\Sigma_q^3)$ described in Proposition~\ref{fix pts}, i.e.\ $y^2z\mapsto \zeta_1 \xi\mapsto a $, $xyz\mapsto \zeta_0^l\zeta_1 \xi\mapsto b$ and $ x^2z \mapsto \zeta_0^{2l} \xi\mapsto c_-$.

The algebra $\cO(\Sigma^3_q(l;-))$ is spanned by  elements of the type $x^ry^sz^t$,  $x^{* r}y^{s}z^{t}$, where $r,s\in \NN$ and  $t\in \ZZ$.
Applying the coaction $\varphi$ to these basis elements gives
$
x^ry^sz^t \mapsto x^ry^sz^t \otimes u^{r+s-2t}. 
$
 Hence $x^ry^sz^t $ is $\varphi$-invariant if and only if $2t =r+s$. If $r$ is even, then $s$ is even and
$$
x^ry^sz^t = x^r y^s z^{{(r+s)}/2} = (x^2 z)^{{r}/2}(y^2 z)^{{s}/2}.
$$
If $r$ is odd, then so is $s$ and 
$$
x^ry^sz^t = x^r y^s z^{{(r+s)}/2} \sim  (x^2 z)^{(r-1)/2}(y^2 z)^{({s-1})/2}(xyz).
$$
The case of $x^{* r}y^{s}z^{t}$ is dealt with similarly, thus proving that all coinvariants  of $\varphi$ are polynomials in $x^2 z$, $xyz$, $y^2z$ and their $*$-conjugates. 
\end{proof}

The main result of this section is contained in the following theorem.
\begin{theorem}\label{thm.main}
$\cO(\Sigma_q^3(l;-))$ is a non-cleft principal $\cO(U(1))$-comodule algebra over $\cO(\RR \PP_q(l;+))$ via the coaction $\varphi$.
\end{theorem}

\begin{proof}
To prove that $\cO(\Sigma_q^3(l;-))$ is a  principal $\cO(U(1))$-comodule algebra over $\cO(\RR \PP_q(l;+))$ we employ Proposition~\ref{prop str con} and construct a strong connection form as follows.

Define $\omega: \cO(U(1)) \rightarrow \cO(\Sigma_q^3(l;-)) \otimes \cO(\Sigma_q^3(l;-))$ recursively as follows.
\begin{subequations}\label{strong-conn}
\begin{equation}
\omega(1)=1 \otimes 1 \label{strong-conn i} 
\end{equation}
\begin{equation}
\omega(u^n)=x^* \omega(u^{n-1})x- \sum_{m=1}^l (-1)^m q^{-m(m+1)} \binom l m_{q^{-2}} y^{2m-1}z^m \omega(u^{n-1})y  \label{strong-conn ii} 
\end{equation}
\begin{equation}
\omega(u^{-n})=x \omega(u^{-n+1})x^*- \sum_{m=1}^l (-1)^m q^{m(m-1)} \binom l m_{q^{2}} y^{2m-1}z^{m-1} \omega(u^{-n+1})yz,  \label{strong-conn iii}
\end{equation}
\end{subequations}
where $n\in \NN$ and, for all $s \in \RR$, the {\em deformed} or {\em q-binomial} coefficients $\binom  l m_{\! s}$  are defined by the following polynomial equality in indeterminate $t$
\begin{equation}\label{binom}
\prod_{m=1}^l (1+s^{m-1}t)=\sum_{m=0}^l s^{m(m-1)/2} \binom l m_{s} t^m.
\end{equation}
The map $\omega$ has been designed such that normalisation property \eqref{strong1} is automatically satisfied. To check property \eqref{strong2} from equations (\ref{strong-conn ii}) and (\ref{strong-conn iii}) take a bit more work. We use proof by induction,
but first have to derive an identity to assist with the calculation.
Set $s=q^{-2}$, $t=-q^{-2}y^* y$ in \eqref{binom} to arrive at
$$
\sum_{m=1}^l (-1)^m q^{-m(m+1)} \binom l m_{q^{-2}} y^{* m}y^m= \prod_{m=1}^l (1+q^{-2(m-1)}(-q^{-2}y^*y))-1,
$$
which, using \eqref{R1},  simplifies to
\begin{equation}
\sum_{m=1}^l (-1)^m q^{-m(m+1)} \binom l m_{q^{-2}} y^{2m} z^m= \prod_{m=1}^l (1-q^{-2m}y^2 z)-1 = x^*x-1. \label{identity}
\end{equation}
Now to start the induction process we consider the case $n=1$. By \eqref{identity} $(\mu \circ \omega)(u)=1$ providing the basis. Next, we assume that the relation holds for $n=N$, that is $(\mu \circ \omega)(u^N)=1$, and consider the case $n=N+1$,
$$
\omega(u^{N+1})=x^* \omega(u^N)x-\sum_{m=1}^l (-1)^mq^{-m(m+1)}\binom l m_{q^{-2}}y^{2m-1}z^m \omega(u^N)y,
$$
applying the multiplication map to both sides and using the induction hypothesis,
$$
(m \circ \omega)(u^{N+1}) =x^*x-\sum_{m=1}^l (-1)^mq^{-m(m+1)} \binom l m_{q^{-2}}y^{2m}z^m 
=x^*x-(x^*x-1)=1,
$$
showing property \eqref{strong2} holds for all $u^n \in \cO(U(1))$, where $n \in \NN$. To show this property holds for each $u^{*n}= u^{-n}$ we adopt the same strategy; this is omitted from the proof as it does not hold further insight, instead repetition of similar arguments. 

Property \eqref{strong3}: this is again  proven by induction. Applying $(\id \otimes \varphi)$ to  $\omega(u)$  gives
\begin{eqnarray*}
x^* \otimes x \otimes u  &-& \sum_{m=1}^l (-1)^m q^{-m(m-1)} \binom l m_{q^2} y^{2m-1} z^m \otimes y \otimes u \\
&=&(x^* \otimes x - \sum_{m=1}^l (-1)^m q^{-m(m-1)} \binom l m_{q^2} y^{2m-1} z^m \otimes y) \otimes u \\
&=&\omega(u) \otimes u=(\omega \otimes \id)  \circ \Delta (u).
\end{eqnarray*}
This shows that property \eqref{strong3} holds for equation (\ref{strong-conn ii}) when $n=1$. We now assume the property holds for $n=N-1$, hence $(id \otimes \varphi) \circ w(u^{N-1})=(\omega \otimes\id) \circ \Delta(u^{N-1})=\omega(u^{N-1}) \otimes u^{N-1}$, and consider the case $n=N$.
\begin{eqnarray*}
(\id \otimes \varphi) (w(u^N))&=&(\id \otimes \varphi)(x^* \omega(u^{N-1})x-\sum_{m=1}^l (-1)^m q^{-m(m-1)} \binom l m_{q^{-2}} y^{2m-1} z^m \omega(u^{N-1})y) \\
&=&x^* ((id \otimes \varphi) (\omega(u^{N-1})x))\\
&& -\sum_{m=1}^l (-1)^m q^{-m(m-1)} \binom l m_{q^{-2}} y^{2m-1} z^m ((id \otimes \varphi)(\omega(u^{N-1})y) \\
&=&x^*\omega(u^{N-1})x \otimes u^{N}- \sum_{m=1}^l (-1)^m q^{-m(m-1)} \binom l m_{q^{-2}} y^{2m-1} z^m \omega(u^{N-1}) y \otimes u^{N} \\
&=& \omega(u^{N}) \otimes u^{N}
= (\omega \otimes \id) \circ \Delta(u^{N}),
\end{eqnarray*}
hence property \eqref{strong3} is satisfied for all $u^n \in \cO(U(1))$ where $n \in \NN$. The case for $u^{*n}$ is proved in a similar manner, as is property \eqref{strong4}. Again, the details are omitted as the process is identical. This completes the proof that $\omega$ is a strong connection form, hence $\cO(\Sigma_q^3(l,-))$ is a principal comodule algebra.

Following the discussion of Section~\ref{section.zz}, to determine whether the constructed comodule algebra is cleft we need to identify invertible elements in $\cO(\Sigma_q^3(l,-))$. 
Since
$$
\cO(\Sigma_q^3(l,-)) \subset \cO(\Sigma_q^3) \cong \cO(S_q^2) \square_{\CC\ZZ_2} \cO(U(1)) \subset \cO(S_q^2) \otimes \cO(U(1)),
$$
and the only invertible elements in the algebraic tensor $\cO(S_q^2) \otimes \cO(U(1))$ are scalar multiples of $1 \otimes u^n$ for $n \in \NN$, we can conclude that the only invertible elements in $\cO(S_q^2) \square_{\CC\ZZ_2}\cO(U(1))$ are the elements of the form $1 \otimes u^n$. These elements correspond to the elements $\xi^n$ in $\cO(\Sigma_q^3)$, which in turn correspond to $z^n$ in $\cO(\Sigma_q^3(l,-))$.

Suppose $j:H \rightarrow A$ is the cleaving map; to ensure the map is convolution invertible we are forced to put $u \mapsto z^n$. Since $u$ has degree $1$ in $H=\cO(U(1))$ and $z$ has degree $-2$ in $\cO(\Sigma_q^3(l,-))$, the map $j$ fails to preserve the degrees hence it is not colinear. Therefore, $\cO(\Sigma_q^3(l,-))$ is a non-cleft principal comodule algebra.
\end{proof}

\subsubsection{~Almost freeness of the coaction $\varrho_{1,l}$.}

At the classical limit, $q\to 1$, the algebras $\cO(\RR \PP_q(l;-))$ represent singular manifolds or orbifolds. It is known that every orbifold can be obtained as a quotient of a manifold by an {\em almost free} action. The latter means that the action has finite (rather than trivial as in the free case) stabiliser groups. As explained in Section~\ref{sec.rev} on the algebraic level freeness is encoded in the bijectivity of the canonical map $\can$, or, more precisely, in the surjectivity of the lifted canonical map $\overline{\can}$  \eqref{bar can}. The surjectivity of $\overline{\can}$ means the triviality of the cokernel of $\overline{\can}$, thus the size of the cokernel of $\overline{\can}$ can be treated as a measure of the size of the stabiliser groups. This leads to the following notion proposed in \cite{BrzFai:tea}.
 
\begin{definition}\label{def.almost.free}
Let $H$ be a Hopf algebra and let $A$ be a right $H$-comodule algebra with coaction $\roA: A\to A\ot H$. We say that the coaction is {\em almost free} if the cokernel of  the (lifted) canonical map
$$
\overline{\can}: A\ot  A\to A\ot H, \qquad a\ot a'\mapsto a\roA(a'),
$$
is finitely generated as a left $A$-module.
\end{definition}

Although the coaction $\varphi$ defined in the preceding section is free, at the classical limit $q\to 1$ $\cO(\Sigma_q^3(l,-))$ represents a singular manifold or an orbifold. On the other hand, at the same limit, $\cO(\Sigma_q^3)$ corresponds to a genuine manifold, one of the Seifert three-dimensional non-orientable manifolds; see \cite{Sco:geo}. It is therefore natural to ask, whether the coaction $\varrho_{1,l}$ of $\cO(U(1))$ on $\cO(\Sigma_q^3)$ which has $\cO(\RR \PP_q(l;-))$ as fixed points is almost free in the sense of Definition~\ref{def.almost.free}.

\begin{proposition}\label{prop.almost0}
The coaction $\varrho_{1,l}$ is almost free.
\end{proposition}
\begin{proof}
Denote by
$
\iota_-: \cO(\Sigma_q^3(l,-)) \hookrightarrow \cO(\Sigma_q^3),$  the $*$-algebra embedding described in Proposition~\ref{fix pts}. One easily checks that the following
 diagram 
$$
\xymatrix{\cO(\Sigma_q^3(l,-)) \ar[rr]^{\iota_-}\ar[d]_{\varphi} && \cO(\Sigma_q^3) \ar[d]^{\varrho_{1,l}} \\
\cO(\Sigma_q^3(l,-))\ot\cO(U(1)) \ar[rr]^-{\iota_-\ot(-)^l} &&\cO(\Sigma_q^3)\ot\cO(U(1)),}
$$
where $(-)^l: u\to u^l$, is commutative. The principality or freeness of $\varphi$ proven in Theorem~\ref{thm.main} implies that $1\ot u^{ml} \in \im (\overline{\can})$, $m\in \ZZ$, where $\overline{\can}$ is the (lifted) canonical map corresponding to coaction $\varrho_{1,l}$. This means that $\cO(\Sigma_q^3)\ot\CC[u^l,u^{-l}] \subseteq \im (\overline{\can})$. Therefore, there is a short exact sequence of left $\cO(\Sigma_q^3)$-modules
$$
\xymatrix{(\cO(\Sigma_q^3)\ot\CC[u,u^{-1}])/(\cO(\Sigma_q^3)\ot\CC[u^l,u^{-l}])\ar[r] & \coker (\overline{\can}) \ar[r] & 0.}
$$
The left $\cO(\Sigma_q^3)$-module $ (\cO(\Sigma_q^3))\ot\CC[u,u^{-1}])/(\cO(\Sigma_q^3)\ot\CC[u^l,u^{-l}])$ is finitely generated, hence so is $\coker (\overline{\can})$.
\end{proof}

\subsubsection{~Associated modules or sections of line bundles.}

One can construct modules associated to the principal comodule algebra $\cO(\Sigma_q^3(l,-))$ following the procedure outlined at the end of Section~\ref{sec.non.pri}; see Definition~\ref{def ass mod}. 

Every one-dimensional comodule of $\cO(U(1)) = \CC[u,u^*]$ is determined by the grading of a basis element of $\CC$, say $1$. More precisely, for any integer $n$, $\CC$ is a left $\cO(U(1))$-comodule with the coaction
$$
\varrho_n: \CC \to\CC[u,u^*]  \ot  \CC, \qquad 1\mapsto u^n\ot 1.
$$
Identifying $\cO(\Sigma_q^3(l,-))\ot \CC$ with $\cO(\Sigma_q^3(l,-))$ we thus obtain, for each coaction $\varrho_n$
$$
\Gamma[n] := \cO(\Sigma_q^3(l,-))\Box_{\cO(U(1))} \CC \cong \{f\in \Sigma_q^3(l,-)\; |\; \varphi(f) = f\ot u^n\} \subset \cO(\Sigma_q^3(l,-)).
$$
In other words, $\Gamma[n]$ consists of all elements of $\cO(\Sigma_q^3(l,-))$ of $\ZZ$-degree $n$. In particular $\Gamma[0]=\cO(\RR \PP_q(l;-))$. Each of the $\Gamma[n]$ is a finitely generated projective left $\cO(\RR \PP_q(l;-))$-module, i.e.\ it represents the module of sections of the non-commutative line bundle over $\RR \PP_q(l;-)$. The idempotent matrix $E[n]$ defining $\Gamma[n]$ can be computed explicitly from a strong connection form $\omega$ (see equations \eqref{strong-conn}) in the proof of Theorem~\ref{thm.main}) following the procedure described in \cite{BrzHaj:Che}. Write $\omega(u^n) = \sum_i  \omega(u^n)\su 1 _i \ot \omega(u^n)\su 2_i$. Then 
\begin{equation}\label{projectors}
E[n]_{ij} = \omega(u^n)\su 2_i\omega(u^n)\su 1_j \in \cO(\RR\PP_q^2(l;-)).
\end{equation}
For example, for $l=2$ and $n=1$, using \eqref{strong-conn ii} and \eqref{strong-conn i} as well as redistributing numerical coefficients we obtain
\begin{equation}\label{e1}
E[1]= \begin{pmatrix}
(1-a)(1-q^2a) & q^{-1}\sqrt{1+q^{-2}}\ b & iq^{-3}ba \cr
q^{-1}\sqrt{1+q^{-2}}\ b^* & q^{-2}(1+q^{-2})\ a & iq^{-4}\sqrt{1+q^{-2}}\ a^2 \cr
iq^{-3} b^* & iq^{-4}\sqrt{1+q^{-2}}\ a & -q^{-6} a^2
\end{pmatrix}
\end{equation}
Although the matrix $E[1]$ is not hermitian, the left-upper $2\times 2$ block is hermitian. On the other hand, once $\cO(\RR \PP_q(2;-))$ is completed to the $C^*$-algebra $C(\RR \PP_q(2;-))$ of continuous functions on $\RR \PP_q(2;-)$ (and then identified with the suitable pullback of two algebras of continuous functions over the quantum real projective space; see \cite{Brz:sei}), then a hermitian projector can be produced out of $E[1]$ by using the Kaplansky formula; see \cite[page~88]{GraVar:ele}.

The traces of tensor powers of each of the $E[n]$ make up a cycle in the cyclic complex  of $\cO(\RR \PP_q(l;-))$, whose corresponding class in  the cyclic homology $HC_\bullet(\cO(\RR \PP_q(l;-)))$ is known as the {\em Chern character} of $\Gamma[n]$. Again, as an illustration of the usage of an explicit form of a strong connection form, we compute the traces of $E[n]$ for general $l$.

\begin{lemma}\label{lemma.cher}
The zero-component of the Chern character of $\Gamma[n]$ is the class of the polynomial $c_n$ in generator $a$ of $\cO(\RR \PP_q(l;-))$, given by the following recursive formula. First, $c_0(a) =1$, and then, for all positive $n$,
\begin{subequations}
\begin{equation}\label{cher+}
c_n(a) = c_{n-1}\left(q^{2l}a\right) \prod_{p=0}^{l-1}\left(1-q^{2p}a\right) +c_{n-1}(a)\left(1-\prod_{p=1}^{l}\left(1-q^{-2p}a\right)\right),
\end{equation}
\begin{equation}\label{cher-}
c_{-n}(a) = c_{-n+1}\left(q^{-2l}a\right) \prod_{p=1}^{l}\left(1-q^{-2p}a\right) +c_{-n+1}(a)\left(1-\prod_{p=0}^{l-1}\left(1-q^{2p}a\right)\right) .
\end{equation}
\end{subequations}
\end{lemma}
\begin{proof}
We will prove the formula \eqref{cher+} as \eqref{cher-} is proven by similar arguments. Recall that $c_n = \tr E[n]$. By normalisation \eqref{strong-conn i} of the strong connection $\omega$, obviously $c_0 =1$. In view of equation \eqref{strong-conn ii} we obtain the following recursive formula
\begin{equation}\label{chern}
c_n = x c_{n-1} x^* - \sum_{m=1}^l (-1)^m q^{-m(m+1)} \binom l m_{q^{-2}} y c_{n-1}y^{2m-1}z^m.
\end{equation}
In principle, $c_n$ could be a polynomial in $a, b$ and $c_-$. However, the third of equations \eqref{R1} together with \eqref{identity} and identification of $a$ as $y^2z$ yield
\begin{equation}\label{cher1}
c_1 =   \prod_{p=0}^{l-1}\left(1-q^{2p}a\right) +\left(1-\prod_{p=1}^{l}\left(1-q^{-2p}a\right)\right),
\end{equation}
that is a polynomial in $a$ only. As commuting $x$ and $y$ through a polynomial in $a$ in formula \eqref{chern} will produce a polynomial in $a$ again, we conclude that each of the $c_n$ is a polynomial in $a$. The second of \eqref{R1}, the centrality of $z$ and the identification of $a$ as $y^2z$ imply that
$$
xc_{n-1}(a) = c_{n-1}(q^{2l}a), \qquad yc_{n-1}(a) = c_{n-1}(a)y,
$$
and in view of \eqref{chern} and \eqref{cher1} yield \eqref{cher+}.
\end{proof}

\subsection{The positive case $\cO(\RR \PP_q(l;+))$.}

\subsubsection{~The principal $\cO(U(1))$-comodule algebra over $\cO(\RR \PP_q^2(l;+))$.}

In the same light as the negative case we aim to construct quantum principal bundles with base spaces $\cO(\RR \PP_q(l;+))$, and procced by viewing $\cO(\Sigma_q^3)$ as a right $H'$-comodule algebra, where $H'$ is a Hopf-algebra of a finite cyclic group. The aim is to construct the total space $\cO(\Sigma_q^3(l,+))$ of the bundle over $\cO(\RR \PP_q(l;+))$ as the coinvariant subalgebra of $\cO(\Sigma_q^3)$. $\cO(\Sigma_q^3(l,+))$ must contain generators $\zeta_1^2\xi$ and $\zeta_0^l\xi$ of $\cO(\RR \PP_q(l;+))$. Suppose $H' = \CC\ZZ_m$  and $\Phi:\cO(\Sigma_q^3) \rightarrow \cO(\Sigma_q^3) \otimes H'$ is a coaction.  We require $\Phi$ to be compatible with the algebraic relations and to give zero $\ZZ_m$-degree to  $\zeta_1^2\xi$ and $\zeta_0^l\xi$ are zero. These requirements yield
$$
 2\deg(\zeta_1)+\deg(\xi)=0\!\!\! \mod \! m, \qquad l\deg (\zeta_0) + \deg(\xi)=0\!\!\!  \mod \! m.
$$
Bearing in mind that $l$ is odd, the simplest solution to these requirements is provided by $m=2l$, $\deg(\xi)=0$, $\deg(\zeta_0) = 2$, $\deg(\zeta_1) =l$. This yields the coaction
$$
\Phi: \cO(\Sigma_q^3) \rightarrow \cO(\Sigma_q^3) \otimes \CC\ZZ_{2l} \qquad\zeta_0 \mapsto \zeta_0 \otimes v^2, \quad \zeta_1 \mapsto \zeta_1 \otimes v^l,\quad  \xi \mapsto \xi \otimes 1 ,
$$
where $v$ ($v^{2l}=1$) is the unitary generator of $\CC\ZZ_{2l}$.
$\Phi$ is extended to the whole of $\cO(\Sigma_q^3)$ so that $\Phi$ is an algebra map, making $\cO(\Sigma_q^3)$ a right $\CC\ZZ_{2l}$-comodule algebra.

\begin{proposition} \label{fix pts A'} The fixed point subalgebra of the coaction $\Phi$  is isomorphic to the $*$-algebra $\cO(\Sigma_q^3(l,+))$ generated by $x', y'$  an central unitary $z'$  subject to the following relations:
\begin{subequations} \label{R'}
\begin{equation}
x'y'=q^{2l} y'x', \qquad x'y'^*=q^{2l} y'^*x', \qquad y'y'^*=q^4y'^*y', \qquad y'^*=y'z'^2,  \label{R'1}
\end{equation}
\begin{equation}
x'x'^*=\prod_{p=0}^{l-1}(1-q^{2p}y'z'), \qquad x'^*x'=\prod_{p=1}^{l}(1-q^{-2p}y'z'). \label{R'2}
\end{equation}
\end{subequations}
The isomorphism between $\cO(\Sigma_q^3(l,+))$ and the coinvariant subalgebra of  $\cO(\Sigma_q^3)$ is given by 
$x'\mapsto \zeta_0^l$, $y'\mapsto \zeta_1^2$ and $z'\mapsto \xi$.
\end{proposition}
\begin{proof}
 Clearly $\zeta_1^2$, $\xi$,  $\zeta_0^{l}$ and  $\zeta_0^{*l}$ are coinvariant elements of $\cO(\Sigma_q^3)$. Apply the coaction $\Phi$ to the basis \eqref{basis.sigma} to obtain
$$
\zeta_0^r \zeta_1^s \xi^t \mapsto \zeta_0^r \zeta_1^s \xi^t\ot v^{2r+ls}, \qquad \zeta_0^{*r} \zeta_1^s \xi^t \mapsto \zeta_0^{*r} \zeta_1^s \xi^t\ot v^{-2r+ls} .
$$ 
These elements are coinvariant, provided $2r+ls=2ml$ in the first case or $-2r+ls=2ml$ in the second. Since $l$ is odd, $s$ must be even and then $r= r'l$, hence the invariant elements must be of the form
$$
(\zeta_0^l)^{r'}(\zeta_1^2)^{s/2}\xi^t, \qquad (\zeta_0^{*l})^{r'}(\zeta_1^2)^{s/2}\xi^t
$$
as required. 
 Relations \eqref{R'} are now easily derived from \eqref{zeta rel} and \eqref{key.rel}.
\end{proof}

The algebra $\cO(\Sigma_q^3(l,+))$ is a right $\cO(U(1))$-comodule with coaction defined as,
\begin{equation}
\Omega: \cO(\Sigma_q^3(l,+))  \rightarrow \cO(\Sigma_q^3(l,+)) \otimes \cO(U(1)), \qquad x' \mapsto x' \otimes u, \quad y' \mapsto y' \otimes u, \quad z' \mapsto z' \otimes u^{-1}.
\end{equation}
The first three relations (\ref{R'1}) bear no information on the possible gradings of the generators of $\cO(\Sigma_q^3(l,+))$, however the final relation of (\ref{R'1}) tells us that the grade of $y'^*$ must have the same grade of $z'$ since, 
\[
\deg(y'^*)=-\deg(y')=\deg(y')+2\deg(z'), 
\] 
{hence}, 
\[
 2\deg(y'^*)=2\deg(z'), \ \text{or}, \ \deg(y'^*)=\deg(z')
\]
This is consistent with relations (\ref{R'2}) since the left hand sides, $x'x'^*$ and $x'^*x'$, have degree zero, as do the right had sides, 
$$
\deg(y'z')=\deg(y')+\deg(y'^*)=\deg(y')+(-\deg(y'))=0
$$
The coaction $\Omega$ is defined setting the grades of $x'$ and $y'$ as 1, and putting the grade of $z'$ as $-1$ to ensure the coaction is compatible with the relations of the algebra $\cO(\Sigma_q^3(l,+))$.

\begin{proposition} The right $\cO(U(1))$-comodule algebra $\cO(\Sigma_q^3(l,+))$ has $\cO(\RR \PP_q(l;+))$ as its subalgebra of coinvariant elements under the coaction $\Omega$.
\end{proposition}

\begin{proof}
The fixed points of the algebra $\cO(\Sigma_q^3(l,+))$ under the coaction $\Omega$ are found using the same method as in the odd $k$ case. A basis for the algebra $\cO(\Sigma_q^3(l,+))$ is given by  $x'^ry'^sz'^t$, $x'^{* r}y'^{s}z'^{t}$, where  $r,s\in \NN$ and $t\in \ZZ$.  

Applying  the coaction $\Omega$ to the first of these basis elements gives,
$$
x'^ry'^sz'^t \mapsto x'^ry'^sz'^t \otimes u^{r+s-t}.
$$ 
Hence the invariance of  $x'^ry'^sz'^t$ is equivalent to $t = r+s$.
Simple substitution and re-arranging gives,
$$
x'^ry'^sz'^t=x'^ry'^{s}z'^{r+s}={(x'z')}^r{(y'z')}^s,
$$
i.e.\ $x'^ry'^sz'^t$ is a polynomial in  $x'z'$ and $y'z'$. 
Repeating the process for the second type of basis element gives the $*$-conjugates of $x'z'$ and $y'z'$. 
Using Proposition~\ref{fix pts A'} we can see that $a=\zeta_1^2 \xi=y'z'$ and $c_+=\zeta_0^l \xi=x'z'$.
\end{proof}

In contrast to the odd $k$ case, although $\cO(\Sigma_q^3(l,+))$ is a principal comodule algebra it yields trivial principal bundle over $\cO(\RR \PP_q(l;+))$.
\begin{proposition}
The right $\cO(U(1))$-comodule algebra $\cO(\Sigma_q^3(l,+))$ is trivial.
\end{proposition}

\begin{proof}
The cleaving map is given by,
$$
j: \cO(U(1)) \rightarrow \cO(\Sigma_q^3(l,+)), \qquad j(u)=z^{' *},
$$
which is an algebra map since $z^{' *}$ is central unitary in $\cO(\Sigma_q^3(l,+))$, hence must be convolution invertible. Also, $j$ is a right $\cO(U(1))$-comodule map since,
$$
(\Omega \circ j)(u)=\Omega(z^{' *})=z^{' *} \otimes u=j(u) \otimes u=(j \otimes\id) \circ \Delta(u),
$$
completing the proof.
\end{proof}
 
Since $\cO(\Sigma_q^3(l,+))$ is a trivial principal comodule algebra, all associated $\cO(\RR \PP_q^2(l;+))$-modules are free.

\subsubsection{~Almost freeness of the coaction $\varrho_{2,l}$.}
As was the case for $\cO(\Sigma_q^3(l,-))$, the principality of $\cO(\Sigma_q^3(l,+))$ can be used to determine that the $\cO(U(1))$-coaction $\varrho_{2,l}$ on $\cO(\Sigma_q^3)$ that defines $\cO(\RR \PP_q^2(l;+))$ is almost free.

\begin{proposition}\label{prop.almost1}
The coaction $\varrho_{2,l}$ is almost free.
\end{proposition}
\begin{proof}
Denote by
$
\iota_+: \cO(\Sigma_q^3(l,+)) \hookrightarrow \cO(\Sigma_q^3),$  the $*$-algebra embedding described in Proposition~\ref{fix pts A'}. One easily checks that the following
 diagram 
$$
\xymatrix{\cO(\Sigma_q^3(l,+)) \ar[rr]^{\iota_+}\ar[d]_{\Omega} && \cO(\Sigma_q^3) \ar[d]^{\varrho_{2,l}} \\
\cO(\Sigma_q^3(l,+))\ot\cO(U(1)) \ar[rr]^-{\iota_+\ot(-)^{2l}} &&\cO(\Sigma_q^3)\ot\cO(U(1)),}
$$
where $(-)^{2l}: u\to u^{2l}$ is commutative. By the arguments analogous to those in the proof of Proposition~\ref{prop.almost0} 
one concludes that there is
a short exact sequence of left $\cO(\Sigma_q^3)$-modules
$$
\xymatrix{(\cO(\Sigma_q^3)\ot\CC[u,u^{-1}])/(\cO(\Sigma_q^3)\ot\CC[u^{2l},u^{-2l}])\ar[r] & \coker (\overline{\can}) \ar[r] & 0,}
$$
where $\overline{\can}$ is the lifted canonical map corresponding to coaction $\varrho_{2,l}$.
The left $\cO(\Sigma_q^3)$-module $ (\cO(\Sigma_q^3)\ot\CC[u,u^{-1}])/(\cO(\Sigma_q^3)\ot\CC[u^{2l},u^{-2l}])$ is finitely generated, hence so is $\coker (\overline{\can})$.
\end{proof}

\section{Conclusions.}

In this paper we discussed the principality of the $\cO(U(1))$-coactions on the coordinate algebra of the quantum Seifert manifold $\cO(\Sigma_q^3)$ weighted by coprime integers $k$ and $l$. We concluded that the coaction is principal if and only if $k=l=1$, which corresponds to the case of a $U(1)$-bundle over the quantum real projective plane. In all other cases the coactions are almost free. We identified subalgebras of $\cO(\Sigma_q^3))$ which admit principal $\cO(U(1))$-coactions, whose invariants are isomorphic to coordinate algebras $\cO(\RR \PP_q^2(l;\pm))$ of quantum real weighted projective spaces. The structure of these subalgebras depends on the parity of $k$. For the odd $k$ case, the constructed principal comodule algebra $\cO(\Sigma_q^3(l,-))$ is non-trivial, while for the even case, the corresponding principal comodule algebra $\cO(\Sigma_q^3(l,+))$ turns out to be trivial.  The triviality of $\cO(\Sigma_q^3(l,+))$ is a disappointment. Whether a different   nontrivial principal $\cO(U(1))$-comodule algebra over $\cO(\RR \PP_q^2(l;+))$ can be constructed or whether such a possibility is ruled out by deeper geometric, topological or algebraic reasons remains to be seen.

\bibliographystyle{mdpi}
\makeatletter
\renewcommand\@biblabel[1]{#1. }
\makeatother

\end{document}